\documentclass[11pt,a4paper]{article}
\usepackage[utf8]{inputenc}
\usepackage{amsmath}
\usepackage{amsfonts}
\usepackage{amssymb}
\usepackage{amsthm}
\usepackage{bbm}
\usepackage{xcolor}
\usepackage{graphicx}
\usepackage{hyperref}
\usepackage{mathtools}
\numberwithin{equation}{section}
\newtheorem{theorem}{Theorem}[section]
\newtheorem{lemma}[theorem]{Lemma}
\newtheorem{definition}[theorem]{Definition}

\newtheorem{proposition}[theorem]{Proposition}
\newtheorem{corollary}{Corollary}[section]

\theoremstyle{definition}
\newtheorem{remark}[theorem]{Remark}
\newtheorem{notation}[theorem]{Notation}

\title{Parameter estimation in rough Bessel model}

\author{Yuliya Mishura$^{1,2}$ \\ \href{mailto:yuliyamishura@knu.ua}{yuliyamishura@knu.ua}  
    \and Anton Yurchenko-Tytarenko$^3$ \\ \href{mailto:antony@math.uio.no}{antony@math.uio.no}}

\date{%
    $^1$Department of Probability, Statistics and Actuarial Mathematics, Taras Shevchenko National University of Kyiv\\%
    $^2$Division of Mathematics and Physics, M\"alardalen University\\
    $^3$Department of Mathematics, University of Oslo
}

\begin{document}

\maketitle

\begin{abstract}
    In this paper, we construct consistent statistical estimators of the Hurst index, volatility coefficient, and drift parameter for Bessel processes driven by fractional Brownian motion with $H<1/2$. As an auxiliary result, we also prove the continuity of the fractional Bessel process. The results are illustrated with simulations.  
\end{abstract}

\section{Introduction}

The Bessel process, defined as the square root $X = \sqrt{Z}$ of the solution $Z = \{Z(t),~t\ge 1\}$ to the stochastic differential equation (SDE)
\begin{align}\label{eq: standard sq Bessel}
    dZ(t) = kdt + 2\sqrt{Z(t)}dW(t), \quad Z(0) = x_0^2 > 0,
\end{align}
is a well-known probabilistic model used in a wide range of fields including physics \cite{Bray2000, Guarnieri_Moon_Wettlaufer_2017, Horibe_Hosoya_Sakamoto_1983} and finance \cite{CIR1981, CIR1985-1, CIR1985-2, Heston_1993}. If $k\in\mathbb N$, $X$ can be interpreted as the Euclidean norm
\begin{equation}\label{eq: Bessel as norm of Wiener}
    X(t) = \sqrt{B^2_1(t) + \cdots + B^2_k(t)}
\end{equation}
of a $k$-dimensional Brownian motion $(B_1,...,B_k)$ that is connected with $W$ in \eqref{eq: standard sq Bessel} via the relation
\[
    W(t) = \sum_{i=1}^k \int_0^t \frac{B_i(s)}{X(s)}dB_i(s),
\]
and therefore the parameter $k$ is often referred to as the \emph{dimension} of the Bessel process $X$. Moreover, as it is shown in \cite[Section 3]{Cherny_2000} (see also \cite{MYT2022}), if $k>1$, the Bessel process $X$ is the unique nonnegative strong solution to the SDE
\begin{equation}\label{eq: Bessel eq}
    dX(t) = \frac{k-1}{X(t)}dt + dW(t), \quad X(0) = x_0>0.
\end{equation}
For more details on various properties of Bessel processes, we refer the reader to \cite[Chapter XI]{Revuz_Yor_1999}, the book \cite{Cherny_Engelbert_2005} which considers general equations of the type \eqref{eq: Bessel eq} or \cite{Bertoin_1990, Bertoin_1990b, Mishura_Pilipenko_Yurchenko-Tytarenko_2023} which deal with the case $0<k<1$.

However, in many applications, standard Brownian motion may not adequately capture the desired level of complexity observed in real-life phenomena. For example, a number of empirical studies \cite{Bollerslev_Mikkelsen_1996, Breidt_Crato_Lima_1998, Cont_2005, Ding_Granger_1996, Ding_Granger_Engle_1993} point out the presence of memory in financial markets. Other sources \cite{Alos_Leon_Vives_2007, Fukasawa_2021, Fukasawa_Takabatake_Westphal_2019, GatheralJaissonRosenbaum2018} indicate that models exhibiting very low H\"older regularity are better suited to reflect the behavior of market volatility. Given that Bessel-type processes are frequently employed in stochastic volatility modeling (see e.g. \cite[Chapter 6]{Jeanblanc_Yor_Chesney_2009}), there is a natural inclination to enhance them by incorporating the aforementioned memory or roughness. A common way to achieve such an effect is to replace (in some sense) the standard Brownian driver $W$ with a fractional Brownian motion $B^H = \{B^H(t),~t\ge 0\}$, i.e. a centered Gaussian process with covariance function
\[
    \mathbb E\left[ B^H(t)B^H(s)\right] = \frac{1}{2}\left(t^{2H} + s^{2H} - |t-s|^{2H}\right), \quad s,t\ge 0.
\] 
For instance, \cite{Essaky_Nualart_2015, Guerra_Nualart_2005, Hu_Nualart_2005} used the property \eqref{eq: Bessel as norm of Wiener} as the starting point of their modification and defined fractional Bessel process $\rho^H$ for $k\in\mathbb N$ as
\[
    \rho^H(t) = \sqrt{(B^H_1(t))^2 + \cdots + (B^H_k(t))^2}
\]
with $B^H_1$,..., $B^H_k$ being $k$ independent fractional Brownian motions. In particular, they prove that $\rho^H$ admits a representation
\begin{equation}\label{eq: Nualart Bessel}
    d\rho^H(t) = H(k-1) \frac{t^{2H-1}}{\rho^H(t)}dt + \sum_{i=1}^k \frac{B^H_i(t)}{\rho^H(t)} dB^H_i(t),
\end{equation}
where the integrals w.r.t. fractional Brownian motions are understood in the divergence sense. 

Another possible ``\textit{fractionalization}'' of the Bessel process may be achieved by replacing $W$ with $B^H$ directly in \eqref{eq: Bessel eq}. This approach was discussed in detail in the series of papers \cite{DNMYT2020, Mishura_Ralchenko_2023, MYT_2018, Mishura_Yurchenko-Tytarenko_2018}: according to it, the fractional Bessel process $X^H = \{X^H(t),~t\ge0\}$ is defined as the a.s. pointwise limit 
\begin{equation}\label{eq: our Bessel process}
    X^H(t) := \lim_{\varepsilon\downarrow 0} X_\varepsilon^H (t)
\end{equation}
of stochastic processes given by the SDE
\[
    X^H_\varepsilon(t) = x_0 +\int_0^t \frac{a}{X^H_\varepsilon(s) \mathbbm 1_{\{X^H_\varepsilon(s)>0\}} + \varepsilon}ds + \sigma B^H(t).
\]

In the present paper, we consider the latter notion of the fractional Bessel process and perform a statistical estimation of parameters $H$, $\sigma$ and $a$ when $H\in\left(0,\frac{1}{2}\right)$. 
\begin{itemize}
    \item In order to estimate $H$ and $\sigma$, we use the standard technique based on quadratic variations of fractional Brownian motion. The main challenge arises in the limit $L^H(t) := \lim_{\varepsilon \downarrow 0}\int_0^t \frac{a}{X^H_\varepsilon(s) \mathbbm 1_{\{X^H_\varepsilon(s)>0\}} + \varepsilon} ds$: when $H<\frac{1}{2}$, \cite[Corollary 3.3]{Mishura_Yurchenko-Tytarenko_2018} only guarantees its continuity in $t$ almost everywhere w.r.t. the Lebesgue measure. It is not enough for the analysis of quadratic variations of $L^H$, so we prove that $L^H$ is continuous at every $t\ge 0$ and utilize this fact to obtain consistent estimators of $H$ and $\sigma$.

    \item Our estimator of $a$ is, in turn, based on the unconventional technique tailored for our specific model. The problem here is that it is currently not known whether the fractional Bessel process exhibits any ergodic properties  typically utilized for drift parameter estimation. In our analysis, we exploit the explicit dynamics of $X^H$ instead to compare the behavior of $X^H(T)$ and $\int_0^T \frac{1}{X^H(t)}dt$ when $T\to\infty$.  
\end{itemize}

The paper is structured as follows. In Section \ref{sec: theory}, we provide the definition of the fractional Bessel process and prove the continuity of the latter. Section \ref{sec: Hurst and volatility} is devoted to the estimation of $H$ and $\sigma$. In Section \ref{sec: drift}, we present our estimator of the drift and prove its consistency. Section \ref{sec: simulations} contains simulations. In Appendix \ref{app: properties of X}, we prove a technical result related to the finiteness of the limit \eqref{eq: our Bessel process}.

\section{Rough Bessel processes}\label{sec: theory}

Let $B^H = \{B^H(t),~t\ge 0\}$ be a fractional Brownian motion with Hurst index $H\in\left(0, \frac 1 2\right)$. 

\begin{remark}
    Using the Kolmogorov-Chentsov theorem, it is possible to prove that $B^H$ has a modification with paths that are locally H\"older continuous up to the order $H$, i.e. for any $T>0$ and $\lambda\in(0,H)$ there exists a positive random variable $\Lambda = \Lambda_{T,\lambda}$ such that
    \begin{equation}\label{eq: BH Holder continuity}
        |B^H(t_1) - B^H(t_2)| \le \Lambda |t_1-t_2|^{\lambda}, \quad t_1,t_2\in[0,T].
    \end{equation}
    Moreover, by \cite{ASVY2014}, the random variable $\Lambda$ can be chosen in such a way that for any $p>0$
    \[
        \mathbb E [\Lambda^p] < \infty.
    \]
    In what follows, we always consider this modification of $B^H$.
\end{remark}
For any $\varepsilon>0$, $x_0>0$, $a>0$, $b\ge 0$ and $\sigma>0$, consider a random process $X^H_\varepsilon = \{X^H_\varepsilon(t),~t\ge 0\}$ given by a stochastic differential equation
\begin{align}\label{eq: pre-limit process}
    X^H_\varepsilon(t) = x_0 +\int_0^t \frac{a}{X^H_\varepsilon(s) \mathbbm 1_{\{X^H_\varepsilon(s)>0\}} + \varepsilon}ds + \sigma B^H(t).
\end{align}
Note that the SDE \eqref{eq: pre-limit process} has Lipschitz continuous drift and additive noise and hence, by the standard Picard iteration argument applied pathwisely, \eqref{eq: pre-limit process} has a unique solution for each $\varepsilon>0$. Moreover, by \cite[Lemma 2.1]{Mishura_Yurchenko-Tytarenko_2018} (see also \cite[Lemma 1.6]{DNMYT2020}), for any $\varepsilon_1 < \varepsilon_2$,
\begin{equation}\label{eq: comparison}
    X^H_{\varepsilon_1}(t) \ge X^H_{\varepsilon_2}(t), \quad t\in[0,T], \quad \omega\in\Omega,
\end{equation}
and hence, for any $t\ge 0$, one can define the limit
\begin{equation}\label{eq: limit process}
    X^H(t) := \lim_{\varepsilon \downarrow 0} X^H_{\varepsilon}(t).
\end{equation}

\begin{definition}
    The process $X^H$ defined by \eqref{eq: limit process} will be called a \emph{fractional} or \emph{rough Bessel process}.
\end{definition}

\begin{remark}\label{rem: properties of RBP}
    This construction of rough Bessel processes was introduced and extensively studied in \cite{Mishura_Yurchenko-Tytarenko_2018}. In particular, it was shown that
    \begin{itemize}
        \item[1)] the limit \eqref{eq: limit process} is finite, i.e., with probability 1,
        \begin{equation}\label{eq: property to reprove}
            X^H(t) < \infty
        \end{equation}
        for all $t\ge 0$;
        \item[2)] with probability 1, $X^H(t) \ge 0$ for all $t\ge 0$ and, moreover, $X^H(t) > 0$ for almost all $t\ge 0$.
    \end{itemize}
    It should be noted that one of the possible cases is missing in Step 2 of the proof of \cite[Theorem 2.1]{Mishura_Yurchenko-Tytarenko_2018} concerning \eqref{eq: property to reprove}. For the reader's convenience, we provide the completed version of it in Appendix \ref{app: properties of X}.
\end{remark}

Since the limit \eqref{eq: limit process} is well-defined and finite, the limit
\begin{equation}
\begin{aligned}
    \lim_{\varepsilon\downarrow 0} \int_0^t \frac{a}{X^H_\varepsilon(s) \mathbbm 1_{\{X^H_\varepsilon(s)>0\}} + \varepsilon}ds &= \lim_{\varepsilon\downarrow 0} \left(  X^H_\varepsilon(t) - x_0 - \sigma B^H(t)\right) 
    \\
    &= X^H(t) - x_0 - \sigma B^H(t)
\end{aligned}
\end{equation}
also exists and is finite. In what follows, we will use the notation
\begin{equation}\label{eq: L}
    L^H(t) := \lim_{\varepsilon\downarrow 0} \int_0^t \frac{1}{X^H_\varepsilon(s) \mathbbm 1_{\{X^H_\varepsilon(s)>0\}} + \varepsilon}ds,
\end{equation}
i.e. $X^H$ satisfies the equation
\begin{equation}\label{eq: equation for X with L}
    X^H(t) = x_0 + a L^H(t) + \sigma B^H(t), \quad t\ge 0.
\end{equation}

By Fatou's lemma, for any $t\ge 0$,
\begin{equation*}
    \int_0^t \frac{1}{X^H(s)}ds \le \liminf_{\varepsilon\downarrow 0} \int_0^t \frac{1}{X^H_\varepsilon(s) \mathbbm 1_{\{X^H_\varepsilon(s)>0\}} + \varepsilon}ds = L^H(t) < \infty.
\end{equation*}
Denote
\[
    R^H(t) := L^H(t) - \int_0^t \frac{1}{X^H(s)}ds
\]
and re-write \eqref{eq: equation for X with L} as
\begin{equation}\label{eq: equation for X}
    X^H(t) = x_0 + a\int_0^t \frac{1}{X^H(s)}ds + \sigma B^H(t) + a R^H(t).
\end{equation}
At the moment, it is not clear whether $R^H(t) = 0$ for all $t\ge 0$. However, we have the following result.

\begin{proposition}
    \begin{itemize}
        \item[1)] With probability 1, there exists $\tau > 0$ such that for all $t\ge \tau$
        \[
            X^H(t) > 0.
        \]
    
        \item[2)] Let $\tau >0$ be such that $X^H(\tau) > 0$. Then there exists a neighbourhood $(\tau_1, \tau_2) \ni \tau$ such that for all $t_1,t_2 \in (\tau_1, \tau_2)$,  $R^H(t_1) - R^H(t_2) = 0$, i.e.
        \begin{equation}\label{eq: L in regular points}
            L^H(t_2) - L^H(t_1) = \int_{t_1}^{t_2} \frac{1}{X^H(s)}ds.
        \end{equation}
    \end{itemize}
\end{proposition}

\begin{proof}
    Item 1) follows directly from \cite[Theorem 3.3]{Mishura_Ralchenko_2023} so let us prove the claim 2). Let $\tau>0$ be such that $X^H(\tau) > 0$. Since $X^H_\varepsilon(\tau) \uparrow X^H(\tau)$ as $\varepsilon \downarrow 0$, there exists $\varepsilon_0 >0$ such that for all $\varepsilon \le \varepsilon_0$, $X^H_\varepsilon(\tau) > 0$. Moreover, the process $X^H_{\varepsilon_0}$ is continuous w.r.t. $t$ hence there exists a neighbourhood $(\tau_1, \tau_2) \ni \tau$ and some positive value $x>0$ such that for all $t\in (\tau_1, \tau_2)$ 
    \[
        X^H_{\varepsilon_0}(t) > x.
    \]
    In particular, \eqref{eq: comparison} implies that for all $\varepsilon \le \varepsilon_0$
    \[
        X^H_{\varepsilon}(t) > x, \quad t\in (\tau_1, \tau_2). 
    \]
    Therefore, for any $\varepsilon \le \varepsilon_0$ and $t \in  (\tau_1, \tau_2)$,
    \begin{align*}
        \frac{1}{X^H_\varepsilon(t) \mathbbm 1_{\{X^H_\varepsilon(t)>0\}} + \varepsilon} & < \frac{1}{x} 
    \end{align*}
    and hence, by the dominated convergence theorem, for any $t_1,t_2 \in(\tau_1,\tau_2)$, $t_1<t_2$,
    \begin{align*}
        L^H(t_2) - L^H(t_1) &= \lim_{\varepsilon\downarrow 0} \int_{t_1}^{t_2} \frac{1}{X^H_\varepsilon(s) \mathbbm 1_{\{X^H_\varepsilon(s)>0\}} + \varepsilon}ds
        \\
        &= \int_{t_1}^{t_2} \frac{1}{X^H(s)}ds.
    \end{align*}
\end{proof}

\begin{corollary}\label{cor: X is Bessel process}
    With probability 1, there exists $\tau > 0$ such that for all $t_2 > t_1 \ge \tau$
    \[
        R^H(t_1) = R^H(t_2)
    \] 
    and
    \[
        L^H(t_2) - L^H(t_1) = \int_{t_1}^{t_2} \frac{1}{X^H(s)}ds.
    \]       
\end{corollary}

Note that \cite[Corollary 3.3]{Mishura_Yurchenko-Tytarenko_2018} establishes only continuity of $L^H$ \textit{almost everywhere} on $\mathbb R_+$ w.r.t. the Lebesgue measure. It is not enough for our purposes: in order to estimate $H$ and $\sigma$, we want to utilize the behavior of quadratic variations of $X^H$ and any possible discontinuities of $L^H$ would create substantial obstacles for our analysis. It turns out, however, that $L^H$ (and hence $X^H$) is continuous at \textit{all} points $t\ge 0$. The corresponding theorem is provided below.

\begin{theorem}\label{th: main property result} \hfill
    \begin{enumerate}
        \item The process $L^H$ defined by \eqref{eq: L} is non-decreasing.
        
        \item The processes $X^H = \{X^H(t),~t\ge 0\}$ and $L^H = \{L^H(t),~t\ge 0\}$ have continuous paths.
    \end{enumerate}
\end{theorem}

\begin{proof}
    First of all, observe that for any $0\le t_1< t_2\le T$
    \begin{align*}
        L^H(t_1) &= \lim_{\varepsilon\downarrow 0} \int_0^{t_1} \frac{1}{X^H_\varepsilon(s) \mathbbm 1_{\{X^H_\varepsilon(s)>0\}} + \varepsilon}ds 
        \\
        & \le \lim_{\varepsilon\downarrow 0} \int_0^{t_2} \frac{1}{X^H_\varepsilon(s) \mathbbm 1_{\{X^H_\varepsilon(s)>0\}} + \varepsilon}ds 
        \\
        & = L^H(t_2),
    \end{align*}
    i.e. the process $L^H$ is indeed non-decreasing. Next, fix an arbitrary deterministic $T>0$, choose an arbitrary $\lambda \in (0,H)$ and let $\Lambda_T$ be such that for all $s,t\in[0,T]$
    \begin{equation}\label{eq: Holder continuity of fBm}
        |B^H(t)-B^H(s)| \le \Lambda|t-s|^\lambda.
    \end{equation}
    
    Note that the monotonicity of $L^H$ implies that its discontinuities can only take the form of positive jumps and, moreover, left and right limits 
    \[
        L^H(t-) := \lim_{\delta \downarrow 0} L^H(t-\delta), \quad L^H(t+) := \lim_{\delta \downarrow 0} L^H(t+\delta)
    \]
    are well-defined at any point $t\in(0,T)$. Next, observe that the limit process $X^H$ defined by \eqref{eq: limit process} satisfies the equation
    \[
        X^H(t) = x_0 + aL^H(t) - b \int_0^t X^H(s)ds + \sigma B^H(t), \quad t\in[0,T],
    \]
    and hence points of discontinuity of $X^H$ coincide with the ones of $L^H$, can occur only in the form of positive jumps of the same size as the corresponding jumps of $L^H$ and, finally, one can define limits
    \[
        X^H(t-) := \lim_{\delta \downarrow 0} X^H(t-\delta), \quad X^H(t+) := \lim_{\delta \downarrow 0} X^H(t+\delta).
    \]
    Assume that $\tau\in[0,T]$ is a point of discontinuity of $L^H$ (and hence $X^H$) for some $\omega\in\Omega$ and observe that $X^H(\tau) = 0$ since otherwise $\tau$ cannot be a point of discontinuity of $X^H$ and $L^H$ by \eqref{eq: L in regular points}. Let $\alpha >0$ be such that
    \[
        X^H(\tau+) = L^H(\tau+) = \alpha
    \]
    and $\delta > 0$ be such that for all $t\in(\tau,\tau+\delta)$ 
    \[
        X^H(t) > \frac{9\alpha}{10}.
    \]
    Next, choose $\delta_0 < \delta$ such that
    \[
        \frac{4a}{\alpha} \delta_0 + \sigma\Lambda_T \delta_0^\lambda \le \frac{\alpha}{4},
    \]
    where $\lambda\in(0,H)$ and $\Lambda_T > 0$ are from \eqref{eq: Holder continuity of fBm}. Since $X^H_\varepsilon(\tau+\delta_0) \uparrow X^H(\tau+\delta_0)$ as $\varepsilon \downarrow 0$, there exists $\varepsilon_0 > 0$ such that for all $\varepsilon \in(0, \varepsilon_0]$
    \[
        X^H_\varepsilon(\tau+\delta_0) > \frac{4\alpha}{5}, \quad X^H_\varepsilon(\tau) \le 0
    \]
    Furthermore, $X^H_\varepsilon(\tau) \uparrow X^H(\tau) = 0$ as $\varepsilon \downarrow 0$ hence $X^H_\varepsilon(\tau) \le 0$ for all $\varepsilon > 0$ and therefore for any $\varepsilon \in(0, \varepsilon_0]$ one can define
    \[
        \tau_\varepsilon^- := \sup\left\{s\in(\tau, \tau+\delta_0):~X^H_\varepsilon(s) = \frac{\alpha}{4}\right\}
    \]
    and
    \[
        \tau_\varepsilon^+ := \sup\left\{s\in(\tau_\varepsilon^-, \tau+\delta_0):~X^H_\varepsilon(s) = \frac{3\alpha}{4}\right\}.
    \]
    By continuity, it is evident that $X_\varepsilon^H(\tau_\varepsilon^-) = \frac{\alpha}{4}$, $X_\varepsilon^H(\tau_\varepsilon^+) = \frac{3\alpha}{4}$ and $X_\varepsilon^H(t) \ge \frac{\alpha}{4}$ for any $t\in[\tau_\varepsilon^-, \tau_\varepsilon^+]$. Therefore for any $\varepsilon \in(0, \varepsilon_0]$
    \begin{align*}
        \frac{\alpha}{2} &= X_\varepsilon^H(\tau_\varepsilon^+) - X_\varepsilon^H(\tau_\varepsilon^-) 
        \\
        &= \int_{\tau_\varepsilon^-}^{\tau_\varepsilon^+} \frac{a}{X^H_\varepsilon(s) \mathbbm 1_{\{X^H_\varepsilon(s)>0\}} + \varepsilon}ds - b \int_{\tau_\varepsilon^-}^{\tau_\varepsilon^+} X_\varepsilon^H(s)ds + \sigma\left(B^H(\tau_\varepsilon^+) - B^H(\tau_\varepsilon^-)\right)
        \\
        &\le \frac{4a}{\alpha} (\tau_\varepsilon^+ - \tau_\varepsilon^-) + \sigma\Lambda_T(\tau_\varepsilon^+ - \tau_\varepsilon^-)^\lambda 
        \\
        &\le \frac{4a}{\alpha} \delta_0 + \sigma\Lambda_T \delta_0^\lambda
        \\
        &<\frac{\alpha}{4},
    \end{align*}
    which gives a contradiction. Therefore, $\alpha = 0$ and $X^H(\tau) = X^H(\tau+)$, i.e. $X^H$ (and hence $L^H$) is right-continuous.

    It remains to notice that $X^H(t-) = X^H(t)$ for any $t\in(0,T]$. Indeed, as mentioned above, $X^H$ is continuous at $t$ if $X^H(t)>0$. If $X^H(t) = 0$,  since $X^H(t)\ge 0$ can only have positive jumps,
    \[
        X^H(t) - X^H(t-) = - X^H(t-)
    \]
    and, since $X^H$ can potentially have only positive jumps, $X^H(t-) = 0$.
\end{proof}

Note that the pre-limit processes $X^H_\varepsilon$ given by \eqref{eq: pre-limit process} are monotonically non-decreasing w.r.t. $\varepsilon$. Moreover, with probability 1, their pointwise limit is continuous function  by Theorem \ref{th: main property result}. Therefore, by Dini's theorem, we immediately obtain uniform convergence which is summarized in the following corollary.

\begin{corollary}
    For any $T>0$,
    \[
        \sup_{t\in[0,T]} |X^H(t) - X^H_\varepsilon(t)| \to 0 \quad\text{a.s.}, \quad \varepsilon \downarrow 0.
    \]
\end{corollary}

\section{Estimation of Hurst index and volatility coefficient}\label{sec: Hurst and volatility}

Let us now move on to the parameter estimation for the rough Bessel processes. Our first goal is to estimate the Hurst index $H \in \left(0,\frac{1}{2}\right)$ and the volatility parameter $\sigma > 0$. Note that, by Theorem \ref{th: main property result}, the process $L^H$ is continuous and non-decreasing and hence has zero quadratic variation. Therefore, we can utilize the standard estimation technique based on power variations. For more details on this method, we refer the reader to \cite{Kubilius_Mishura_Ralchenko_2017}. 

Throughout this section, we assume that the rough Bessel process $X^H = \{X^H(t),~t\ge 0\}$ satisfies equation \eqref{eq: equation for X} with (unknown) parameters $a,\sigma>0$, $H<1/2$ and is observed on a discrete uniform partition $0=t_0<t_1<...<t_n = T$ of a fixed compact $[0,T]$, $t_k := \frac{kT}{n}$.

\subsection{Estimation of $H$}

Let us start with some useful notation for quadratic variations.

\begin{notation}
    For a stochastic process $\xi = \{\xi(t),~t\in[0,T]\}$ observed on a discrete uniform partition $0=t_0<t_1<...<t_n = T$, $t_k := \frac{kT}{n}$, denote 
    \begin{equation}\label{eq: 1st variation}
        V^n_{1,2}(\xi) := \sum_{k=0}^{n-1} (\xi(t_{k+1}) - \xi(t_k))^2
    \end{equation}
    and
    \begin{equation}\label{eq: 2nd variation}
        V^n_{2,2}(\xi) := \sum_{k=0}^{n-2} (\xi(t_{k+2}) - 2\xi(t_{k+1}) + \xi(t_k))^2.
    \end{equation}
\end{notation}

Note that the zero-mean Gaussian sequence $\{B^H(n) - B^H(n-1), n\ge 1\}$, is stationary and it is easy to check that
\begin{align*}
    \mathbb E\left[B^H(1)\left(B^H(n) - B^H(n-1)\right)\right] &= \frac{1}{2} \left( n^{2H} - 2(n-1)^{2H} +(n-2)^{2H} \right) 
    \\
    &= O(n^{2H-2})\to 0
\end{align*}
as $n\to\infty$. Hence $\{B^H(n) - B^H(n-1), k\ge 1\}$ is ergodic and, by the ergodic theorem, with probability 1,
\[
    \frac{1}{n}\sum_{k=0}^{n-1} (B^H(k+1) - B^H(k))^2 \to \mathbb E[(B^H(1))^2] = 1
\]
and
\[
    \frac{1}{n}\sum_{k=0}^{n-2} (B^H(k+2) - 2B^H(k+1) + B^H(k))^2 \to \mathbb E[(B^H(2) - 2B^H(1))^2] = 4-2^{2H}
\]
as $n\to\infty$. Furthermore, by the self-similarity property of $B^H$, 
\[
    \frac{n^{-1+2H}}{T^{2H}}\sum_{k=0}^{n-1} \left(B^H(t_{k+1}) - B^H(t_{k})\right)^2 \stackrel{\text{law}}{=} \frac{1}{n}\sum_{k=0}^{n-1} (B^H(k+1) - B^H(k))^2,
\]
\begin{align*}
    \frac{n^{-1+2H}}{T^{2H}} &\sum_{k=0}^{n-2} (B^H(t_{k+2}) - 2B^H(t_{k+1}) + B^H(t_{k}))^2 
    \\
    &\stackrel{\text{law}}{=} \frac{1}{n}\sum_{k=0}^{n-2} (B^H(k+2) - 2B^H(k+1) + B^H(k))^2,
\end{align*}
and we have the following result.

\begin{theorem}\label{th: k-p-variation consistency}
    Let $H\in(0,1)$ and $\sigma>0$. Then, as $n\to\infty$,
    \[
        \left(\frac{n}{T}\right)^{-1+2H} V^n_{1,2}(\sigma B^H) \xrightarrow{\mathbb P} \sigma^2 T
    \]
    and
    \[
        \left(\frac{n}{T}\right)^{-1+2H} V^n_{2,2}(\sigma B^H) \xrightarrow{\mathbb P}  (4-2^{2H})\sigma^2 T.
    \]
\end{theorem}

\begin{remark}
    A more general result related to pathwise integrals w.r.t. $B^H$ can be found in \cite[Theorem 1]{Corcuera_Nualart_Woerner_2006}. More detailed information on the convergence rates can be found in \cite[Chapter 2]{Kubilius_Mishura_Ralchenko_2017}.
\end{remark}

As established in Theorem \ref{th: main property result}, the process $a L^H$ in the right-hand side of \eqref{eq: equation for X} is a continuous process of bounded variation and hence, with probability 1,
\[
     V^n_{1,2}(aL^H) \to 0, \quad V^n_{2,2}(aL^H) \to 0, \quad n\to\infty.
\]
Therefore, we immediately get the following corollary.

\begin{corollary}\label{cor: quadratic variation of Bessel process}
    Let $X^H$ be a rough Bessel process given by \eqref{eq: equation for X} observed on a discrete uniform partition $0=t_0<t_1<...<t_n = T$, $t_k = \frac{kT}{n}$. Then, as $n\to\infty$,
    \[
        \left(\frac{n}{T}\right)^{-1+2H} V^n_{1,2}(X^H) \xrightarrow{\mathbb P} \sigma^2 T
    \]
    and
    \[
        \left(\frac{n}{T}\right)^{-1+2H} V^n_{2,2}(X^H) \xrightarrow{\mathbb P} (4-2^{2H})\sigma^2 T.
    \]
\end{corollary}

Next, define
\begin{equation}\label{eq: estimator of H}
    \widehat H := \frac{\log\left( 4 - \frac{V^n_{2,2}(X^H)}{V^n_{1,2}(X^H)} \right)}{2\log 2}.
\end{equation}

\begin{theorem}
    The estimator $\widehat H$ given by \eqref{eq: estimator of H} is (weakly) consistent estimator of the Hurst index $H$, i.e.
    \[
        \widehat H \xrightarrow{\mathbb P} H, \quad n \to \infty.
    \]
\end{theorem}
\begin{proof}
    The statement follows immediately from Corollary \ref{cor: quadratic variation of Bessel process} and the continuous mapping theorem.
\end{proof}

\begin{remark}
    By Corollary \ref{cor: quadratic variation of Bessel process},
    \[
        4 - \frac{V^n_{2,2}(X^H)}{V^n_{1,2}(X^H)} \xrightarrow{\mathbb P} 2^{2H} > 1, \quad n\to\infty,
    \]
    and hence, in practice, the logarithm in \eqref{eq: estimator of H} is well-defined for large enough values of $n$. Moreover,
    \begin{equation}\label{eq: Cauchy}
    \begin{aligned}
        V^n_{2,2}(X^H) &= \sum_{k=0}^{n-2} (X^H(t_{k+2}) - 2X^H(t_{k+1}) + X^H(t_k))^2
        \\
        & = \sum_{k=0}^{n-2} \left((X^H(t_{k+2}) - X^H(t_{k+1})) - (X^H(t_{k+1}) - X^H(t_k))\right)^2
        \\
        &\le 2 \sum_{k=0}^{n-2} (X^H(t_{k+2}) - X^H(t_{k+1}))^2 + 2 \sum_{k=0}^{n-2} (X^H(t_{k+1}) - X^H(t_k))^2
        \\
        & \le 4V^n_{1,2}(X^H),
    \end{aligned}
    \end{equation}
    i.e.
    \[
        4 - \frac{V^n_{2,2}(X^H)}{V^n_{1,2}(X^H)} \ge 0,
    \]
    and the equality in \eqref{eq: Cauchy} occurs only if $x_0 = X^H(t_0) = X^H(t_1) = ... = X^H(t_n)$. It remains an open question whether such an event can occur with positive probability but, in practice, it is never the case.
\end{remark}

\subsection{Estimation of $\sigma$}

Next, assume that the Hurst index $H$ is known and the goal is to estimate the volatility coefficient $\sigma$. Using Corollary \ref{cor: quadratic variation of Bessel process}, it can be done in a straightforward manner. Namely, we have the following result.

\begin{theorem}
    Let the Hurst index $H$ be known. Then the estimator
    \begin{equation}\label{eq: sigma estimator}
        \widehat\sigma = \frac{1}{T^H}\sqrt{ n^{-1+2 H} V^n_{1,2}(X^H)}
    \end{equation} is (weakly) consistent estimator of the volatility coefficient $\sigma$, i.e.
    \[
        \widehat \sigma \xrightarrow{\mathbb P} \sigma, \quad n \to \infty.
    \]
\end{theorem}
\begin{proof}
    The statement follows immediately from Corollary \ref{cor: quadratic variation of Bessel process} and the continuous mapping theorem.
\end{proof}

\section{Drift parameter estimation}\label{sec: drift}

Let us move to the estimation of the drift parameter $a$ in \eqref{eq: equation for X}. In contrast to the high-frequency setting considered in Section \ref{sec: Hurst and volatility}, we will now assume that we observe a continuous path of $X^H$ on $[0,T]$ with $T \to \infty$. 
We study the following estimator of the parameter $a$
\begin{equation}\label{eq: estimator of a 1}
    \widehat a(T) := \frac{X^H(T)}{ \int_0^T \frac{1}{X^H(t)}dt}.
\end{equation}

\begin{remark}
    Note that we do not assume any prior knowledge of the parameters $H$ or $\sigma$.
\end{remark}

Let us start by presenting a well-known result related to the growth of fractional Brownian motion (for more details, see \cite{Kozachenko_Melnikov_Mishura_2015}, \cite[Section B 3.5]{Kubilius_Mishura_Ralchenko_2017} or \cite[Corollary 2.2]{Mishura_Ralchenko_2023}).

\begin{theorem}\label{th: growth of fBm}
    For any $\delta > 0$ there exist random time $\tau_\delta$ and positive random variable $\xi$  such that, with probability 1,
    \begin{equation}\label{eq: growth of BH}
        \max_{s\in[0,T]} |B^H(t)| \le \xi T^{H+\delta}
    \end{equation}
    for all $T>\tau_\delta$.
\end{theorem}

\begin{lemma}\label{lemma: growth of 1 over X}
    Let $X^H$ be a rough Bessel process satisfying \eqref{eq: equation for X}. Then, with probability 1,
    \begin{equation}\label{eq: asymptotics 1}
        \liminf_{T\to\infty} \frac{1}{\sqrt{T}} \int_0^T \frac{1}{X^H(t)}dt > 0.
    \end{equation}
\end{lemma}

\begin{proof}
    Take an arbitrary $\delta\in\left(0, \frac{1}{2} - H\right)$ and fix an $\omega\in\Omega$ such that the corresponding path of $B^H$ is continuous, \eqref{eq: growth of BH} holds and $R^H$ becomes constant after some random time point $\tau$ as described in Corollary \ref{cor: X is Bessel process}.

    Assume that \eqref{eq: asymptotics 1} does not hold, i.e. there exists an increasing sequence $\{T_n,~n\ge 1\}$ such that $T_n\to \infty$ as $n\to\infty$, $T_1$ exceeds $\tau_\delta$ from Theorem \ref{th: growth of fBm}, and
    \[
        \frac{1}{\sqrt{T_n}} \int_0^{T_n} \frac{1}{X^H(t)}dt \to 0, \quad n\to \infty.
    \]
    Integrating both sides of \eqref{eq: equation for X} from $0$ to $T_n$, we obtain that for all $n\ge 1$
    \begin{equation}\label{proofeq: integrated fractional Bessel process}
    \begin{aligned}
        \int_0^{T_n} X^H(t) dt &= x_0 T_n + a \int_0^{T_n} \int_0^t \frac{1}{X^H(s)}ds dt 
        \\
        &\quad +  \sigma \int_0^{T_n} B^H(t) dt + a \int_0^{T_n} R^H(t) dt.
    \end{aligned}
    \end{equation}
    Let us prove that the equality \eqref{proofeq: integrated fractional Bessel process} cannot hold by comparing the asymptotics of its left- and right-hand sides.
    
    On the one hand, using the Cauchy-Schwartz inequality, it is easy to see that for all $T>0$
    \[
        \int_0^T X^H(t)dt \int_0^T \frac{1}{X^H(t)}dt \ge T^2,
    \]
    so
    \[
        \frac{1}{T^{3/2}_n} \int_0^{T_n} X^H(s)ds \to \infty, \quad n \to \infty.
    \]
    On the other hand,
    \[
        \frac{x_0 T_n}{T^{3/2}_n} \to 0, \quad n \to \infty,
    \]
    \begin{align*}
        \frac{1}{T^{3/2}_n} \int_0^{T_n} \int_0^t \frac{1}{X^H(s)}ds dt &= \frac{1}{T^{3/2}_n} \int_0^{T_n} \frac{1}{X^H(t)} (T_n - t)dt
        \\
        & \le \frac{T_n}{T^{3/2}_n}  \int_0^{T_n} \frac{1}{X^H(t)} dt 
        \\
        &\to 0 , \quad n \to \infty,
    \end{align*}
    and, by \eqref{eq: growth of BH},
    \begin{align*}
        \frac{1}{T^{3/2}_n} \left| \int_0^{T_n} B^H(t) dt \right| & \le \frac{T^{1+H+\delta}}{T^{3/2}_n} \frac{\xi}{1+H+\delta} \to 0, \quad n \to \infty.
    \end{align*}
    Finally, since $R^H$ is constant after some time point $\tau$,
    \[
        \frac{1}{T^{3/2}_n} \int_0^{T_n} R^H(t) dt \to 0, \quad n\to \infty.
    \]
    In other words, if such sequence $\{T_n,~n\ge 1\}$ exists, the left-hand side of \eqref{proofeq: integrated fractional Bessel process} divided by $T_n^{3/2}$ converges to $\infty$ whereas the right-hand side divided by $T_n^{3/2}$ converges to zero. We get a contradiction that proves \eqref{eq: asymptotics 1}.
\end{proof}

We are now ready to move to the main result of this section.

\begin{theorem}
    With probability 1,
     \[
        \widehat a(T) \to a, \quad T\to\infty.
    \]
\end{theorem}

\begin{proof}
    First of all, note that, by Lemma \ref{lemma: growth of 1 over X}, 
    \[
        \frac{x_0}{\int_0^T \frac{1}{X^H(s)}ds} \to 0 \quad \text{a.s.}
    \]
    when $T \to \infty$. Next, since there exists $\tau >0$ such that $R^H(t) = 0$ for all $t\ge \tau$, Lemma \ref{lemma: growth of 1 over X} also implies that
    \[
        \frac{a R^H(t)}{\int_0^T \frac{1}{X^H(s)}ds} \to 0 \quad \text{a.s.}
    \]
    when $T \to \infty$. Finally, fixing $\delta \in \left(0, \frac{1}{2} - H\right)$ and using \eqref{eq: growth of BH}, we can deduce that
    \begin{equation*}
    \begin{aligned}
        \frac{|B^H(T)|}{\int_0^T \frac{1}{X^H(s)}ds} & \le \frac{1}{T^{\frac{1}{2} - H - \delta}}\frac{\xi}{\frac{1}{\sqrt{T}} \int_0^T \frac{1}{X^H(s)}ds } \to 0 \quad \text{a.s.}
    \end{aligned}    
    \end{equation*}
    as $T\to \infty$. Therefore, with probability 1,
    \begin{equation*}
    \begin{aligned}
        \widehat a_1(T) &= \frac{x_0 + a\int_0^t \frac{1}{X^H(s)}ds + \sigma B^H(t) + a R^H(t)}{ \int_0^T \frac{1}{X^H(s)}ds }
        \\
        & = a + \frac{x_0}{\int_0^T \frac{1}{X^H(s)}ds} + \frac{\sigma B^H(t)}{\int_0^T \frac{1}{X^H(s)}ds} + \frac{a R^H(t)}{\int_0^T \frac{1}{X^H(s)}ds}
        \\
        & \to a, \quad T\to\infty.
    \end{aligned}   
    \end{equation*}  
\end{proof}

\section{Simulations}\label{sec: simulations}

In this section, we illustrate our results with simulations. In all cases, we use the standard Euler scheme to simulate the pre-limit processes $X^H_\varepsilon$ given by \eqref{eq: pre-limit process} with $\varepsilon = 0.0001$. Each estimator was tested on 1000 samples. All simulations were performed using \textsf{R} programming language; in order to simulate trajectories of fractional Brownian motion, the package \textsf{somebm} was used.

\subsection{Simulation of $\widehat H$ and $\widehat \sigma$}\label{subsec: sim H and sigma}

We start by testing the estimators $\widehat H$ and $\widehat \sigma$ given by \eqref{eq: estimator of H} and \eqref{eq: sigma estimator} respectively. The real values of the parameters are chosen to be $H=0.3$ and $\sigma = 1$ whereas the time horizon is $T=1$. 

At first, we analyze $\widehat H$ for different sizes $n$ of the partition. For each $n$, 1000 trajectories of $X^H$ were generated and the value of $\widehat H$ was computed for each of the generated paths. Table \ref{table: H} contains mean, variance and coefficient of variation (CV, i.e. the ratio of the standard deviation to the mean) for $\widehat H$ and Figure  \ref{fig: H} depicts the corresponding box-and-whisker plots.

\begin{table}[h!]
\centering
\begin{tabular}{|c|c|c|c|c|}
\hline
         & $n=100$ & $n=1000$ & $n=10000$ & $n=100000$ \\ \hline
Mean     & 0.3074612      &   0.3018724     &  0.299684       &   0.2999879     \\
Variance &  0.007608759     &  0.0006979747      & 0.00007438351        &  0.000007596678      \\
CV       &  0.2837047     &  0.08751781      & 0.02877894        &  0.009187727      \\ \hline
\end{tabular}
\caption{Performance of the estimator $\widehat H$ given by \eqref{eq: estimator of H} for different sizes of the partition. The real value is $H=0.3$.}\label{table: H}
\end{table}

\begin{figure}[h!]
    \centering
    \includegraphics[width = \textwidth]{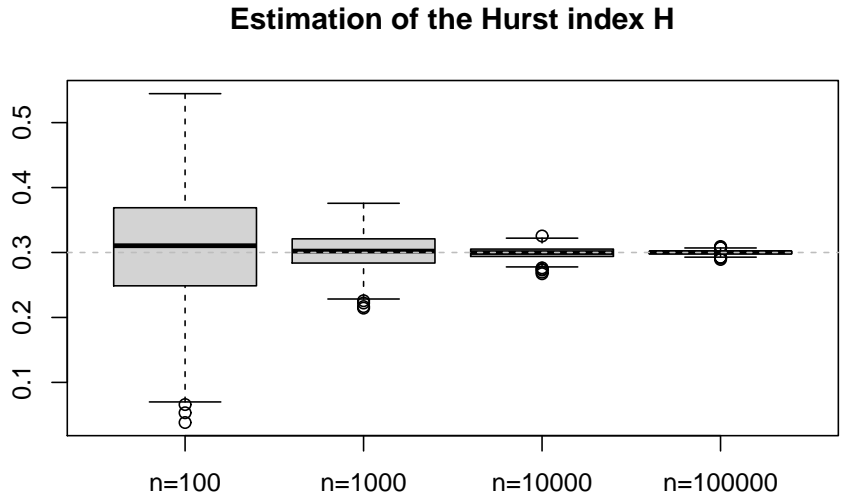}
    \caption{Performance of the estimator $\widehat H$ given by \eqref{eq: estimator of H} for different sizes of the partition. The real value $H=0.3$ is depicted with a grey dashed line.}\label{fig: H}
\end{figure}

Next, we perform the same procedure for the estimator $\widehat \sigma$ given by \eqref{eq: sigma estimator} under assumption that the real value of $H$ is known. The results are given in Table \ref{table: sigma} and Figure \ref{fig: sigma}

\begin{table}[h!]
\centering
\begin{tabular}{|c|c|c|c|c|}
\hline
         & $n=100$ & $n=1000$ & $n=10000$ & $n=100000$ \\ \hline
Mean     &  0.9964649     & 0.9999765       & 0.9996666        &  0.9999794      \\
Variance & 0.005895953      & 0.0006022594       &  0.00005529735      & 0.000005777412       \\
CV       & 0.07705752      & 0.02454155       & 0.007438699        & 0.002403674       \\ \hline
\end{tabular}
\caption{Performance of the estimator $\widehat \sigma$ given by \eqref{eq: sigma estimator} for different sizes of the partition under assumption that $H$ is known. The real value is $\sigma=1$.}\label{table: sigma}
\end{table}

\begin{figure}[h!]
    \centering
    \includegraphics[width = \textwidth]{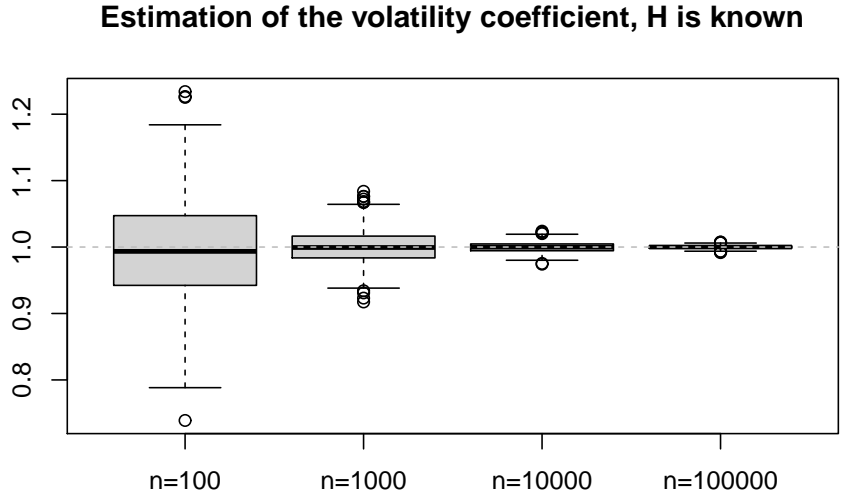}
    \caption{Performance of the estimator $\widehat \sigma$ given by \eqref{eq: sigma estimator} for different sizes of the partition under assumption that $H$ is known. The real value $\sigma=1$ is depicted with a grey dashed line.}\label{fig: sigma}
\end{figure}

Finally, we consider a more realistic situation when neither $H$ nor $\sigma$ are known. In this case, we first estimate $H$ using the estimator $\widehat H$ given by \eqref{eq: estimator of H} and then plug in the result in the estimator $\widehat\sigma$ defined by \eqref{eq: sigma estimator}. The results of this estimation of $\sigma$ are given in Table \ref{table: sigma H} and Figure \ref{fig: sigma H}  

\begin{table}[h!]
\centering
\begin{tabular}{|c|c|c|c|c|}
\hline
         & $n=100$ & $n=1000$ & $n=10000$ & $n=100000$ \\ \hline
Mean     & 1.098972      & 1.021549       & 0.9991251        &  1.001703      \\
Variance & 0.1783677      & 0.03501052       & 0.006181774        & 0.000860674       \\
CV       & 0.3843009      & 0.183164       & 0.0786931        & 0.02928738   \\ \hline
\end{tabular}
\caption{Performance of the estimator $\widehat \sigma$ given by \eqref{eq: sigma estimator} for different sizes of the partition; $H$ is unknown and estimated using \eqref{eq: estimator of H}. The real value is $\sigma=1$.}\label{table: sigma H}
\end{table}

\begin{figure}[h!]
    \centering
    \includegraphics[width = \textwidth]{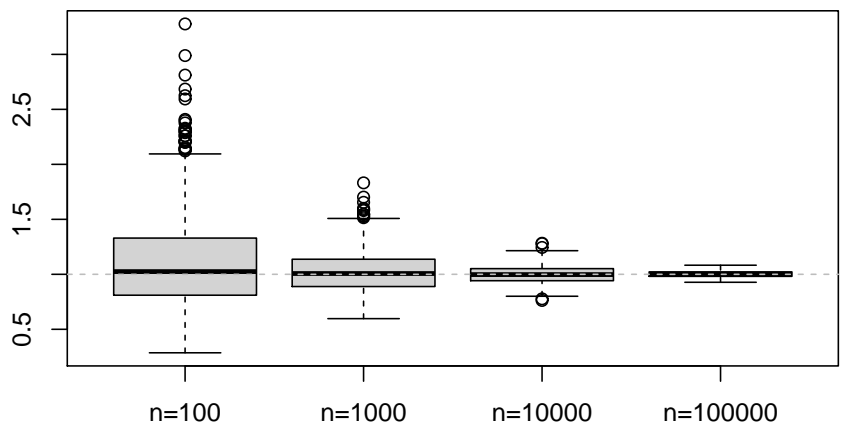}
    \caption{Performance of the estimator $\widehat \sigma$ given by \eqref{eq: sigma estimator} for different sizes of the partition; $H$ is unknown and estimated using \eqref{eq: estimator of H}. The real value $\sigma=1$  is depicted with a grey dashed line.}\label{fig: sigma H}
\end{figure}

Overall, in all cases, simulations show convergence of the estimators to the true values. This also applies to the least favorable situation of estimation of $\sigma$ when $H$ is unknown.

\subsection{Simulation of $\widehat a(T)$}

Next, we test the performance of the estimator $\widehat a(T)$ for different time horizons $T$ using the same numerical techniques as in Subsection \ref{subsec: sim H and sigma}. The real values of the parameters are once again chosen to be $H=0.3$, $\sigma = 1$ and the integral $\int_0^T \frac{1}{X^H(t)}dt$ is approximated by the integral sum 
\[
    \frac{1}{n}\sum_{i=0}^{[Tn]-1} \frac{1}{\max\left\{X^H\left(\frac{i}{n}\right), 0.001\right\}}
\]
with $n=10000$. For each $T$, 1000 simulations  of $\widehat a(T)$ were performed. The results are given in Table \ref{table: a} and Figure \ref{fig: a}

\begin{table}[h!]
\centering
\begin{tabular}{|c|c|c|c|c|}
\hline
         & $T=1$ & $T=10$ & $T=100$ & $T=1000$ \\ \hline
Mean     &  3.872642     &  2.409692      & 2.109389        & 2.029289       \\
Variance &  4.09064     &  0.5914294      &  0.1749705       & 0.06310528       \\
CV       &  0.5222618     &  0.3191463      & 0.1983014        &  0.1237909      \\ \hline
\end{tabular}
\caption{Performance of the estimator $\widehat a(T)$ for different time horizons $T$. The real value is $a=2$.}\label{table: a}
\end{table}

\begin{figure}[h!]
    \centering
    \includegraphics[width = \textwidth]{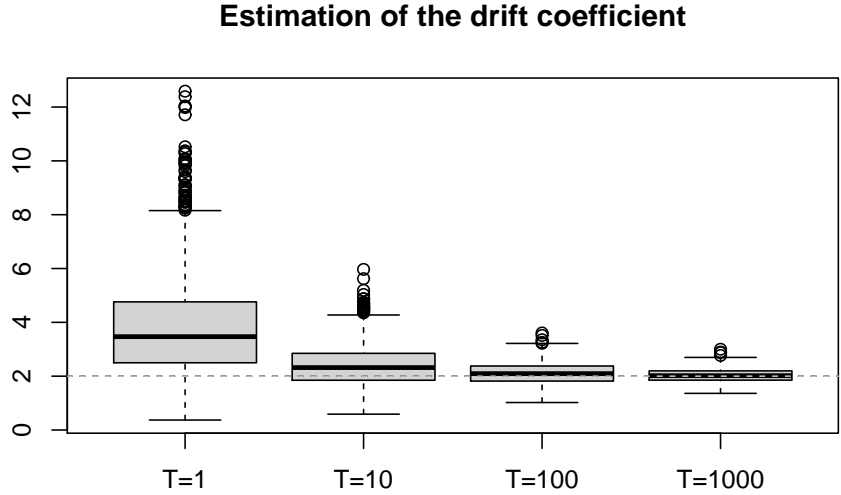}
    \caption{Performance of the estimator $\widehat a(T)$ for different time horizons $T$. The real value $a=2$  is depicted with a grey dashed line.}\label{fig: a}
\end{figure}

Simulations confirm the convergence of $\widehat a(T)$ to $a$. However, they also indicate that relatively high values of $T$ are required to guarantee reasonable variance of the estimator.

\section{Acknowledgements}

The present research is carried out within the frame and support of the ToppForsk project nr. 274410 of the Research Council of Norway with title STORM: Stochastics for Time-Space Risk Models. The first  author is supported by The Swedish Foundation for Strategic Research, grant Nr. UKR22-0017 and by Japan Science and Technology Agency CREST, project reference number JPMJCR2115.

\appendix

\section{Finiteness of rough Bessel processes}\label{app: properties of X}

As mentioned in Remark \ref{rem: properties of RBP}, the goal of this appendix is to prove the following result.

\begin{theorem}\label{th: appendix theorem}
    Let $B^H$ be a fractional Brownian motion with $H<1/2$ and $X^H$ be defined by \eqref{eq: limit process}. Then, with probability 1, for any $t\ge 0$
    \begin{equation}\label{eq: appendix finiteness of X}
        X^H(t) < \infty.
    \end{equation}
\end{theorem}

Before moving to the proof of Theorem \ref{th: appendix theorem}, let us first give some auxiliary results.

\begin{proposition}\label{prop: appendix prop X1}
    Let $X^H_1 = \{X^H_1(t),~t\ge 0\}$ be a stochastic process defined by 
    \[
        X^H_1(t) = x_0 +\int_0^t \frac{a}{X^H_1(s) \mathbbm 1_{\{X^H_1(s)>0\}} + 1}ds + \sigma B^H(t).
    \]
    Fix $T>0$, $\lambda \in (0,H)$ and let $\Lambda = \Lambda_{T,\lambda}$ be a random variable such that \eqref{eq: BH Holder continuity} holds. Then, with probability 1,
    \[
        \sup_{t\in[0,T]} |X^H_1(t)| \le \left(x_0 + aT\right) + \sigma T^{\lambda } \Lambda.
    \]
\end{proposition}

\begin{proof}
    Fix $\omega\in\Omega$ such that the corresponding path $t \mapsto B^H(\omega, t)$ is H\"older continuous. In particular, 
    \[
        \max_{t\in[0,T]}|B^H(\omega, t)| \le \Lambda T^\lambda.
    \]
    Taking into account that $\frac{a}{x\mathbbm 1_{\{x>0\}} + 1} \le a$ for any $x\in \mathbb R$,  we can write
    \begin{align*}
        |X^H_1(t)| & \le x_0 + \int_0^t a ds + \sigma T^{\lambda } \Lambda
        \\
        &\le \left(x_0 + aT\right) + \sigma T^{\lambda } \Lambda.
    \end{align*}
\end{proof}

\begin{proposition}\label{prop: finiteness of pre-limit}
    Let $X^H_\varepsilon$ be defined by \eqref{eq: pre-limit process}. Fix $T>0$, $\lambda \in (0,H)$ and let $\Lambda = \Lambda_{T,\lambda}$ be a random variable such that \eqref{eq: BH Holder continuity} holds. Then, for any $\varepsilon \in(0,1]$ and $t\in[0,T]$,
        \begin{equation}\label{eq: finiteness}
            |X^H_{\varepsilon}(t)| < x_0 + aT \left(\frac{2}{x_0} \vee 1\right) + \sigma T^\lambda \Lambda \quad \text{a.s.}
        \end{equation}
\end{proposition}
\begin{proof}
    Fix $\varepsilon\in(0,1]$ and take $\omega\in\Omega$ such that the corresponding path $t \mapsto B^H(\omega, t)$ is H\"older continuous. In particular,
    \[
        \max_{t\in[0,T]}|B^H(\omega, t)| \le \Lambda T^\lambda,
    \]
    Denote also
    \[
        \tau_{1,\varepsilon} := \sup\left\{t\in [0,T]~\Big|~\forall s\in[0,t]:~X^H_\varepsilon(s) \ge \frac{x_0}{2}\right\}.
    \]
    Our goal is to prove \eqref{eq: finiteness} separately for $t\in[0,\tau_{1,\varepsilon}]$ and $t\in[\tau_{1,\varepsilon},T]$.

    \textbf{Step 1:} $t\in[0,\tau_{1,\varepsilon}]$. In this case, $X^H_{\varepsilon}(s) \ge \frac{x_0}{2}$ for all $s\in[0,t]$ and hence we can write
    \begin{align*}
        |X^H_{\varepsilon}(t)| & = \left|x_0 +\int_0^t \frac{a}{X^H_\varepsilon(s) \mathbbm 1_{\{X^H_\varepsilon(s)>0\}} + \varepsilon}ds + \sigma B^H(t)\right|
        \\
        &\le  x_0 + \int_0^t \frac{a}{X^H_\varepsilon(s) \mathbbm 1_{\{X^H_\varepsilon(s)>0\}} + \varepsilon}ds + \sigma T^{\lambda} \Lambda
        \\
        & \le  x_0 + a \int_0^t\frac{2}{x_0}ds + \sigma T^{\lambda } \Lambda 
        \\
        &\le  \left(x_0 + \frac{2aT}{x_0}\right) + \sigma T^{\lambda} \Lambda
        \\
        & \le \left(x_0 + aT \left(\frac{2}{x_0} \vee 1\right)\right) + \sigma T^\lambda \Lambda
    \end{align*}
    as required.

    \textbf{Step 2:} $t\in[\tau_{1,\varepsilon},T]$. Consider
    \[
        \tau_{2,\varepsilon}(t) :=  \sup\left\{s\in [\tau_{1,\varepsilon},t]~\Big|~|X^H_\varepsilon(s)| < \frac{x_0}{2}\right\}.
    \]
    If $\tau_{2,\varepsilon}(t) = t$, \eqref{eq: finiteness} holds automatically, so let us assume that $\tau_{2,\varepsilon}(t) < t$. In this case, by continuity, $X^H_{\varepsilon}(\tau_{2,\varepsilon}(t)) = \frac{x_0}{2}$ and, moreover, $|X^H_{\varepsilon}(s)| \ge \frac{x_0}{2}$ for all $s\in[\tau_{2,\varepsilon}(t), t]$ and we have two possibilities.
    \begin{itemize}
        \item[1)]  If $X^H_{\varepsilon}(s) \le -\frac{x_0}{2}$ for all $s\in[\tau_{2,\varepsilon}(t), t]$, then
        \[
            X^H_1(s) \le X^H_{\varepsilon}(s) < 0
        \]
        and hence, by Proposition \ref{prop: appendix prop X1},
        \begin{equation}\label{proofeq: appendix case 2 1}
            |X^H_{\varepsilon}(s)| \le |X^H_1(s)| \le \left(x_0 + aT\right) + \sigma T^{\lambda } \Lambda.
        \end{equation}

        \item[2)] If $X^H_{\varepsilon}(s) \ge \frac{x_0}{2}$ for all $s\in[\tau_{2,\varepsilon}(t), t]$, then, just like in Case 1 above, 
        \begin{equation}\label{proofeq: appendix case 2 2}
        \begin{aligned}
            |X^H_{\varepsilon}(t)| & \le \left|X^H_{\varepsilon}(\tau_{2,\varepsilon}(t))\right| +\left|\int_{\tau_{2,\varepsilon}(t)}^t \frac{a}{X^H_\varepsilon(s) \mathbbm 1_{\{X^H_\varepsilon(s)>0\}} + \varepsilon}ds\right|
            \\
            &\qquad  + \sigma \left|B^H(t) - B^H(\tau_{2,\varepsilon}(t))\right|
            \\
            &\le \frac{x_0}{2} + a \int_{\tau_{2,\varepsilon}(t)}^t \frac{2}{x_0} ds + \sigma T^\lambda \Lambda
            \\
            & \le \left(\frac{x_0}{2} + \frac{2aT}{x_0} \right) + \sigma T^\lambda \Lambda.
        \end{aligned}
        \end{equation}
    \end{itemize}
    Hence, summarizing \eqref{proofeq: appendix case 2 1} and \eqref{proofeq: appendix case 2 2}, for any $t\in[\tau_{1,\varepsilon},T]$
    \[
        X^H(t) \le x_0 + aT \left(\frac{2}{x_0} \vee 1\right) + \sigma T^\lambda \Lambda,
    \]
    which ends the proof.
\end{proof}

\begin{proof}[Proof of Theorem \ref{th: appendix theorem}]
    By definition, for any $t\ge 0$
    \[
        X_{\varepsilon}(t) \uparrow X^H(t) \quad a.s., \quad \varepsilon \downarrow 0,
    \]
    so \eqref{eq: appendix finiteness of X} follows immediately from Proposition \ref{prop: finiteness of pre-limit}.
\end{proof}

\bibliographystyle{acm}
\bibliography{biblio}

\end{document}